\documentclass[10pt]{article}

\usepackage[english]{babel} 
\usepackage{amssymb,amsmath,amsthm}
\usepackage{graphicx}
\usepackage [applemac] {inputenc}
\usepackage{geometry}      
\usepackage{color,graphics}
\usepackage[T1]{fontenc}
\usepackage{dsfont}
\usepackage{tikz}
\usepackage{float}

\title{Resolution of an integral equation with the Thue-Morse sequence}
\author{
               Jean-Fran\c{c}ois Bertazzon \\ 
               \small{ Laboratoire d’Analyse, Topologie et Probabilités,} \\
               \small{Aix-Marseille Université,} \\
               \small{Avenue de l'escadrille Normandie-Niémen. 13397  Marseille, France} 
}

\makeatletter
\renewcommand\section{\@startsection{section}{1}{\z@}{6pt\@plus0pt}{6pt}{\bfseries\scshape\centering}}
\renewcommand\subsection{\@startsection {subsection}{2}{\z@}{6pt \@plus 0ex  \@minus 0ex}{-6pt \@plus 0pt}{\reset@font\bfseries \ \ \ \ \ \ \ \ }}
\makeatother

\makeatletter
\renewenvironment{thebibliography}[1]
     {\section*{\refname}%
      \@mkboth{\MakeUppercase\refname}{\MakeUppercase\refname}%
      \list{\@biblabel{\@arabic\c@enumiv}}%
           {\settowidth\labelwidth{\@biblabel{#1}}%
            \leftmargin\labelwidth
            \advance\leftmargin\labelsep
            \@openbib@code
            \usecounter{enumiv}%
            \let\p@enumiv\@empty
            \itemsep=0pt
            \parsep=0pt
            \leftmargin=\parindent
            \itemindent=-\parindent
            \renewcommand\theenumiv{\@arabic\c@enumiv}}%
      \sloppy
      \clubpenalty4000
      \@clubpenalty \clubpenalty
      \widowpenalty4000%
      \sfcode`\.\@m}
     {\def\@noitemerr
       {\@latex@warning{Empty `thebibliography' environment}}%
      \endlist}
\makeatother

\begin{document}

\newcommand{\floor}[1]{{\left\lfloor #1 \right\rfloor}}

\newtheorem{lemma}{Lemma}
\newtheorem{theorem}{Theorem}
\maketitle

\begin{abstract}
It is a classical fact that the exponential function is solution of the integral equation $ \int _0^X   f(x)dx + f(0) =f(X)$. If we slightly modify this equation to  $ \int _0^X   f(x)dx+f(0)=f(\alpha X)$ with $\alpha\in ]0,1[$, it seems that no classical techniques apply to yields solutions. In this article, we consider the parameter $\alpha=1/2$. We will show the existence of a solution wich takes the values of the Thue-Morse sequence on the odd integers.
\end{abstract}

\section{Introduction}

We consider the functional equation  
\begin{equation} \label{eq:0}  \int \limits_{0} ^{X}  \ f(x)dx \ +\  f(0) \ = \ f \left( \frac{X}{2} \right) . \end{equation}

We can see that the set of continuous solutions is a closed vector space,  containing the identically zero function. It is quite clear that any continuous function satisfying Equation (\ref{eq:0}) is differentiable infinitely many times. So, Equation (\ref{eq:0}) can be rewritten $f(X) = f' \left( X/2 \right)/2$.

We can easily verify that the nonzero solutions cannot be expanded in a series. In addition, two solutions equal in a neighborhood of $0$ are equal everywhere. 

We let $\tau$ denote the Thue-Morse substitution. It is a morphism of the free monoid generated by $-1$ and $1$,  defined by $\tau (-1) = (-1) 1$ and $\tau (1) = 1 (-1)$ and let $\boldsymbol{u} = (u_n) _ {n \geq 0} = (-1) 11 (-1) 1 (-1) (-1) 1 \dots $  be the Thue-Morse sequence, one of the fixed points of this substitution. See \cite{allouche95b,Berstel92axelthue's,N2} for details.

The aim of this work is to show the following result:

\begin{theorem} \label{th}
There exists a continuous function $f_\infty$ valued in $[-1,1]$, solution of Equation (\ref{eq:0}), such that 

$\bullet$ \hspace{6pt} for each integer $n$, $f_\infty(2n+1)=u_n$ and $f_\infty(2n)=0$;

$\bullet$ \hspace{6pt} for each negative real number $x$, $f_\infty(x)=0$;

$\bullet$ \hspace{6pt} for each positive real number $x$, $|f_\infty(x)|=|f_\infty(x+2)|$.
\end{theorem}

\begin{figure}[H] \centering \includegraphics[width=10cm]{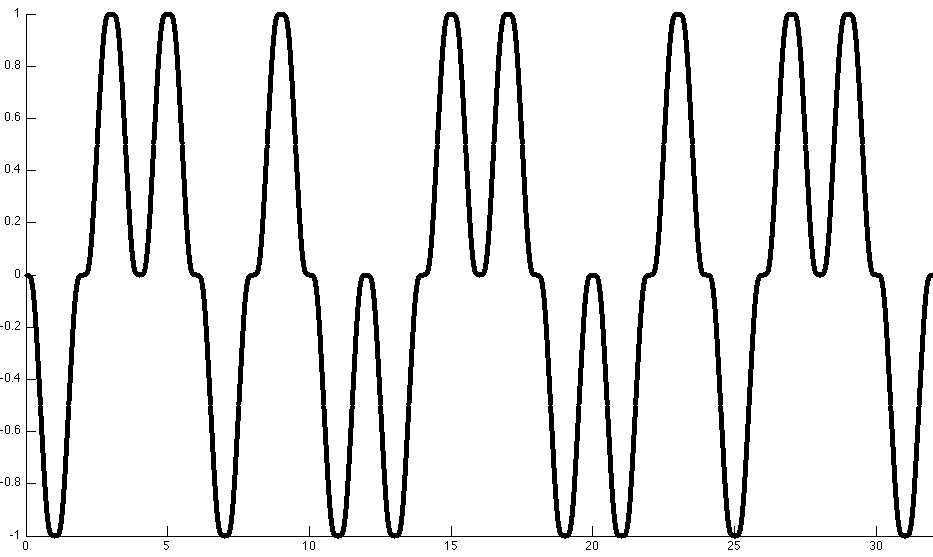} \caption{Representation of the graph of $f_\infty$.} \end{figure}

\section{Introduction of some combinatorial objects} \label{se:introcomb}

For any integers $k\geq0$ and $n\geq1$, we define the quantities $(\Sigma ^ k_n)_{(k,n) \in \mathbb N^2}$  by 
\begin{equation} \label{eq:tes} \Sigma ^k_0 = u_{k} \mbox{ and } \Sigma ^0_n=0, \end{equation}
and by induction for any integers $k\geq0$ and $n\geq0$, by 
\begin{equation} \label{eq:te} \Sigma^{k+1}_{n+1} = \Sigma^{k}_{n}+\Sigma^{k}_{n+1}.  \end{equation}

\begin{figure}[H]
\centering
\begin{tikzpicture}[scale=0.65]
\node [black] (B0) at (-1,-0.5) {$-1$};   \node [black] (B1) at (2,-0.5) {$0$};  \node [black] (B2) at (4,-0.5) {$0$};  \node [black] (B2bis) at (6,-0.5) {$0$};  \node [black] (BB2) at (7,-0.5) {};  \node [black] (B3) at (9,-0.5) {}; 
\node [black] (B4) at (10,-0.5) {$0$};  \node [black] (B5) at (12,-0.5) {$0$};  \node [black] (BB5) at (13,-0.5) {};  \node [black] (B6) at (15,-0.5) {};
\node [black] (C0) at (-1,-1) {$1$};   \node [black] (C1) at (2,-1) {$-1$};  \node [black] (C2) at (4,-1) {$0$};  \node [black] (C2bis) at (6,-1) {$0$};  \node [black] (CC2) at (7,-1) {};  \node [black] (C3) at (9,-1) {}; 
\node [black] (C4) at (10,-1) {$0$};  \node [black] (C5) at (12,-1) {$0$};  \node [black] (CC5) at (13,-1) {};  \node [black] (C6) at (15,-1) {};
\node [black] (D0) at (-1,-1.5) {$1$};   \node [black] (D1) at (2,-1.5) {$0$};  \node [black] (D2) at (4,-1.5) {$-1$};  \node [black] (D2bis) at (6,-1.5) {$0$};  \node [black] (DD2) at (7,-1.5) {};  \node [black] (D3) at (9,-1.5) {}; 
\node [black] (D4) at (10,-1.5) {$0$};  \node [black] (D5) at (12,-1.5) {$0$};  \node [black] (DD5) at (13,-1.5) {};  \node [black] (D6) at (15,-1.5) {};
\node [black] (E0) at (-1,-2) {$-1$};   \node [black] (E1) at (2,-2) {$1$};  \node [black] (E2) at (4,-2) {$-1$}; \node [black] (E2bis) at (6,-2) {$-1$};  \node [black] (EE2) at (7,-2) {}; 
\node [black] (E3) at (9,-2) {};  \node [black] (E4) at (10,-2) {$0$};  \node [black] (E5) at (12,-2) {$0$};  \node [black] (EE5) at (13,-2) {};  \node [black] (E6) at (15,-2) {};
\node [black] (E10) at (-1,-2.5) {$1$};   \node [black] (E11) at (2,-2.5) {$0$};  \node [black] (E12) at (4,-2.5) {$0$};  \node [black] (E12) at (6,-2.5) {$-2$};  \node [black] (E1E2) at (7,-2.5) {}; 
\node [black] (E13) at (9,-2.5) {};  \node [black] (E14) at (10,-2.5) {$0$};  \node [black] (E15) at (12,-2.5) {$0$};  \node [black] (E1E5) at (13,-2.5) {};  \node [black] (E16) at (15,-2.5) {};
\node [black] (FG0) at (-2.73,-4) {$\Sigma^k_0=$};  \node [black] (F0) at (-1,-4) {$u_k$};   \node [black] (FF0) at (2,-4) {$\Sigma^k_1$};   \node [black] (FF0) at (4,-4) {$\Sigma^k_2$}; 
  \node [black] (F2) at (5,-4) {};   \node [black] (F3) at (9,-4) {};  
\node [black] (F4) at (10,-4) {$\Sigma^k_n$};  \node [black] (F5) at (12,-4) {$\Sigma^k_{n+1}$};  \node [black] (FF5) at (13,-4) {}; \node [black] (F6) at (15,-4) {};
\node [black] (GF0) at (-3,-5) {$\Sigma^{k+1}_0=$};  \node [black] (G0) at (-1,-5) {$u_{k+1}$};   \node [black] (G0) at (2,-5) {$\Sigma^{k+1}_1$};  
 \node [black] (G0) at (4,-5) {$\Sigma^{k+1}_2$};   \node [black] (G1) at (5,-5) {};  
\node [black] (G4) at (11,-5) {};   \node [black] (G5) at (12,-5) {$\Sigma^{k+1}_{n+1}$};  \node [black] (GG5) at (13,-5) {}; \node [black] (G6) at (15,-5) {};
\draw[black,densely dashed] (BB2) --    (B3); \draw[black,densely dashed] (CC2) --    (C3); \draw[black,densely dashed] (DD2) --    (D3); \draw[black,densely dashed] (EE2) --    (E3);
\draw[black,densely dashed] (E1E2) --    (E13); \draw[black,densely dashed] (BB5) --    (B6); \draw[black,densely dashed] (CC5) --    (C6); \draw[black,densely dashed] (DD5) --    (D6);
\draw[black,densely dashed] (EE5) --    (E6); \draw[black,densely dashed] (E1E5) --    (E16); \draw[black,densely dashed] (F2) --    (F3); \draw[black,densely dashed] (FF5) --    (F6);
\draw[black,densely dashed] (G1) --    (G4); \draw[black,densely dashed] (GG5) --    (G6); \draw[black,densely dashed] (E10) --    (F0); \draw[black,densely dashed] (E14) --    (F4); \draw[black,densely dashed] (E15) --    (F5);
\end{tikzpicture}
\caption{"Pascal's Triangle"  associated to the Thue-Morse sequence.}
\end{figure}

In \cite{prunescu}, M. Prunescu has studied the behavior of  certain double sequences, called \textit{recurrent two-dimensional sequences} in a more general context. For example when the initialization  of the induction given in Equation (\ref{eq:tes}) is 
\[
\Sigma ^k_0 = v_{k} \mbox{ and } \Sigma ^0_n=w_n,
\]
where $(v_n)_n$ and $(w_n)_n$ are  sequences such that $v_0=w_0$. He is particularly interested in the case where $\boldsymbol{v}=\boldsymbol{w}=\boldsymbol{u}$.
\bigskip

If we cleverly renormalize the lines of the standard Pascal triangle, we can approximate a Gaussian curve. We will renormalize the columns of the Pascal triangle  associated to the Thue-Morse sequence, to approximate the function $f_\infty$. We will see that each column is uniformly bounded. This is a very special property of the Thue-Morse sequence. 

This property does not hold for Sturmian words, for which the sequence $(\Sigma ^k_2)_k$ is not bounded.  More precisely, for each parameter $\alpha\in [0,1]$, we put $\boldsymbol{v}(\alpha)=\big{(}v_n(\alpha)\big{)}_n$ the sequence defined for each integer $n$ by $v_n(\alpha)=\floor{(n+1)\alpha}-\floor{n\alpha}$. We associate to the sequence  $\boldsymbol{v}(\alpha)$ the sequence $\boldsymbol{w}(\alpha)=\big{(}w_n(\alpha)\big{)}_n$ defined for each integer $n$ by $w_n(\alpha) = \alpha$ if $v_n(\alpha)=0$, and $w_n(\alpha) = -(1-\alpha)$ otherwise. So, the sequence $(∑^k_1)_1$ defined in (\ref{eq:te}) associated to the sequence $\boldsymbol{w}(\alpha)$ is bounded. But the sequence $(\Sigma ^k_2)_k$ is not bounded. We refer to \cite{MR2040589}, \cite{MR0113864} and \cite{MR1743497}. 
\bigskip

For all integers $n$,  we define a real function $f_{n}$, by $f_n(x)=0$ if $x\leq 0$, and
\[
f_n(x)= x^k_n +2^{n-1}\delta_x (x^{k+1}_n-x^k_n), \mbox{ if } x= \frac{k}{2^{n-1}} +\delta_x \mbox{ and } 0\leq \delta_x< 2^{1-n}
\]
for an integer $k$, with the notation $x^k_n \ = \  2^{-(n-1)(n-2)/2} \cdot \Sigma ^k_n $.

\begin{figure}[H] \centering \includegraphics[width=14cm]{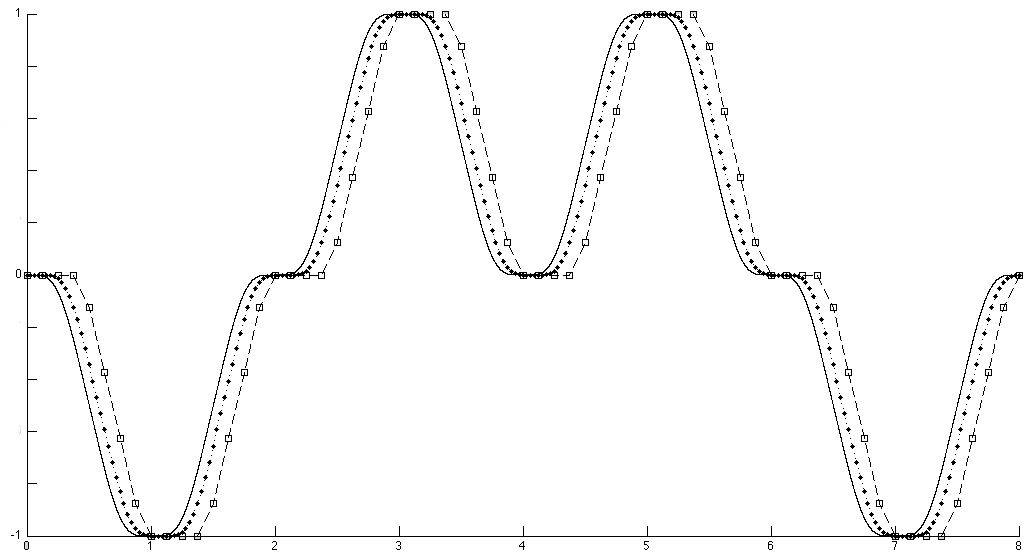} \caption{Representation of graph  of $f_4$, $f_6$ and $f_\infty$.} \end{figure}

We may also approach this problem from a dynamical point of view. We define $T$, the application from the set of real sequences into itself by
\[
T\Big{(}  (y_n)_{n\geq 1} \Big{)}  =  \Big{(}  y_1 ,  y_2+y_1 ,  y_3+\frac{y_2}{2^{1}}  ,   \dots  ,  y_{n+1}+\frac{y_n}{2^{n-1}}  ,  \dots  \Big{)}.
\]
We must then consider  the $n$-th coordinates of the sequence $(\boldsymbol{y}^k)_{k\geq0}$ up to renormalization,  where  $\boldsymbol{y}^k=(y^k_n)_{n\geq1}$, is defined by induction by
$\boldsymbol{y}^0 = \boldsymbol{0}=(0,\dots,0,\dots)$, and for each integer $k\geq 1$,
\[
 \boldsymbol{y}^{k+1}=T(\boldsymbol{y}^k)-(u_{k},0,\dots,0,\dots). 
\]

\section{Calculation of points $x^k_n$  for the first $n$}
\label{se:cal}

We calculate the initial values of sequences $(x^k_n)_{k}$. To do this, we note that for each integer $k$, $u_{2k} = u_{k} =-u_{2k+1}$. 

\[
\left\{
\begin{array}{cllllll} \Sigma ^{2k} _{1} &=&  0, \\
\Sigma ^{2k+1} _{1} &=&  u_{2k}=u_k, \\
\end{array}
\right.
\Longleftrightarrow
\left\{
\begin{array}{cllllll}
x ^{2k} _{1} &=&  0,\\
x^{2k+1} _{1} &=&u_k. \\
 \end{array}
\right.
\]
\[
\left\{
\begin{array}{clllll}
\Sigma ^{4k} _{2} &=&  0,&& \\
\Sigma ^{4k+1} _{2} &=&  \Sigma ^{4k} _{2}+\Sigma ^{4k}_1& =& 0, \\
\Sigma ^{4k+2} _{2} &=&  \Sigma ^{4k+1} _{2}+\Sigma ^{4k+1}_1 &=& u_{2k}=u_k, \\
\Sigma ^{4k+3} _{2} &=&  \Sigma ^{4k+2} _{2}+\Sigma ^{4k+2}_1 &=& u_{2k}=u_k, \\
\end{array}
\right.
\Longleftrightarrow
\left\{
\begin{array}{cllllll}
x ^{4k} _{2} &=&  0,\\
x^{4k+1} _{2} &=&0, \\
x ^{4k+2} _{2} &=& u_{k},\\
x^{4k+3} _{2} &=&u_{k}. \\
 \end{array}
\right.
\]
\[
\left\{
\begin{array}{clllll}
\Sigma ^{8k} _{3} &=&   0, &&\\
\Sigma ^{8k+1} _{3} &=&  \Sigma ^{8k} _{3}+\Sigma ^{8k}_2 &=& 0, \\
\Sigma ^{8k+2} _{3} &=&  \Sigma ^{8k+1} _{3}+\Sigma ^{8k+1}_2 &=& 0, \\
\Sigma ^{8k+3} _{3} &=&  \Sigma ^{8k+2} _{3}+\Sigma ^{8k+2}_2 &=& u_{4k}=u_k, \\
\Sigma ^{8k+4} _{3} &=&  \Sigma ^{8k+3} _{3}+\Sigma ^{8k+3}_2 &=& 2u_{4k}=2u_k, \\
\Sigma ^{8k+5} _{3} &=&  \Sigma ^{8k+4} _{3}+\Sigma ^{8k+4}_2 &=& 2u_{4k}=2u_k, \\
\Sigma ^{8k+6} _{3} &=&  \Sigma ^{8k+5} _{3}+\Sigma ^{8k+5}_2 &=& 2u_{4k}=2u_k, \\
\Sigma ^{8k+7} _{3} &=&  \Sigma ^{8k+6} _{3}+\Sigma ^{8k+6}_2 &=& u_k, \\
\end{array}
\right.
\Longleftrightarrow
\left\{
\begin{array}{cllllll}
x ^{8k} _{3} &=&  0,\\
x^{8k+1} _{3} &=&0, \\
x ^{8k+2} _{3} &=& 0,\\
x^{8k+3} _{3} &=&u_{k}/2, \\
x ^{8k+4} _{3} &=&  u_k,\\
x^{8k+5} _{3} &=&u_k, \\
x ^{8k+6} _{3} &=& u_{k},\\
x^{8k+7} _{3} &=&u_{k}/2. \\
 \end{array}
\right.
\]
\[
\left\{
\begin{array}{clllll}
\Sigma ^{16k} _{4} &=&   0, &&\\
\Sigma ^{16k+1} _{4} &=&  \Sigma ^{16k} _{4}+\Sigma ^{16k}_3 &=& 0, \\
\Sigma ^{16k+2} _{4} &=&  \Sigma ^{16k+1} _{4}+\Sigma ^{16k+1}_3 &=& 0, \\
\Sigma ^{16k+3} _{4} &=&  \Sigma ^{16k+2} _{4}+\Sigma ^{16k+2}_3 &=& 0, \\
\Sigma ^{16k+4} _{4} &=&  \Sigma ^{16k+3} _{4}+\Sigma ^{16k+3}_3 &=&u_k, \\
\Sigma ^{16k+5} _{4} &=&  \Sigma ^{16k+4} _{4}+\Sigma ^{16k+4}_3  &=&3u_k, \\
\Sigma ^{16k+6} _{4} &=&  \Sigma ^{16k+5} _{4}+\Sigma ^{16k+5}_3 &=& 5u_k, \\
\Sigma ^{16k+7} _{4} &=&  \Sigma ^{16k+6} _{4}+\Sigma ^{16k+6}_3 &=& 7u_k, \\
\Sigma ^{16k+8} _{4} &=&  \Sigma ^{16k+7} _{4}+\Sigma ^{16k+7}_3 &=& 8u_k, \\
\Sigma ^{16k+9} _{4} &=&  \Sigma  ^{16k+8} _{4}+\Sigma ^{16k+8}_3 &=& 8u_k, \\
\Sigma ^{16k+10} _{4} &=&  \Sigma ^{16k+9} _{4}+\Sigma ^{16k+9}_3 &=& 8u_k, \\
\Sigma ^{16k+11} _{4} &=&  \Sigma ^{16k+10} _{4}+\Sigma ^{16k+10}_3 &=& 8u_k, \\
\Sigma ^{16k+12} _{4} &=&  \Sigma ^{16k+11} _{4}+\Sigma ^{16k+11}_3 &=& 7u_k, \\
\Sigma ^{16k+13} _{4} &=&  \Sigma ^{16k+12} _{4}+\Sigma ^{16k+12}_3 &=& 5u_k, \\
\Sigma ^{16k+14} _{4} &=&  \Sigma ^{16k+13} _{4}+\Sigma ^{16k+13}_3 &=& 3u_k, \\
\Sigma ^{16k+15} _{4} &=&  \Sigma ^{16k+14} _{4}+\Sigma ^{16k+14}_3 &=& u_k, \\
\end{array}
\right.
\Longleftrightarrow
\left\{
\begin{array}{clll}
x ^{16k} _{4} &=&   0, \\
x ^{16k+1} _{4} &=&  0, \\
x ^{16k+2} _{4} &=&  0, \\
x ^{16k+3} _{4} &=&   0, \\
x ^{16k+4} _{4} &=&  u_k/8, \\
x ^{16k+5} _{4} &=& 3u_k/8, \\
x ^{16k+6} _{4} &=& 5u_k/8, \\
x ^{16k+7} _{4} &=& 7u_k/8, \\
x ^{16k+8} _{4} &=&  u_k, \\
x ^{16k+9} _{4} &=&  u_k, \\
x ^{16k+10} _{4} &=&  u_k, \\
x ^{16k+11} _{4} &=&  u_k, \\
x ^{16k+12} _{4} &=&  7u_k/8, \\
x ^{16k+13} _{4} &=&  5u_k/8, \\
x ^{16k+14} _{4} &=&   3u_k/8, \\
x ^{16k+15} _{4} &=&  u_k/8. \\
\end{array}
\right.
\]

\section{First combinatorial results}

\begin{lemma} \label{le:lle}
For any integers $n\geq 1$,  $k\geq 0$ and  $l\in\{0,\dots,2^n-1\}$, there exists $a(n,l)$, which does not depend on  $k$, such that $\Sigma ^{2^n k+l}_n = a(n,l)u_k$. In particular,  $\Sigma ^{2^n k}_n = a(n,0)=0$. For any integer  $n\geq1$ and $l\in\{0,\dots,2^{n}-1\}$, the coefficients $a(n, l)$ satisfy the following relation:
\begin{equation} \label{eq:1}
\begin{array}{clcl} &a(n+1,l+1) &=& a(n+1,l)+a(n,l) \\ \mbox{ and} &a(n+1,l+2^{n}+1) &= &a(n+1,l+2^n)-a(n,l). \end{array}
\end{equation}
We conclude that $a(n+1,l+2^n) = a(n+1,2^n)-a(n+1,l)$.
\end{lemma}

\begin{proof}
We have seen in Section \ref{se:cal}, that this result is true for the first values  of the integer $n$. We suppose that the result is true up to a rank $n-1$ and we will show that it is still  true up to order $n$. We start by verifying that  $\Sigma_{n+1}^{2^{n}k}$ is zero for each integer $ k $:
\[
\begin{array}{clll}
\Sigma_{n}^{2^{n}k}&=& \sum \limits_{l=0}^{2^{n}k-1} \Sigma_{n-1} ^{l} + \Sigma_{n+1} ^{0} =  \sum \limits_{j=0}^{k-1} \sum \limits_{l=0}^{2^n-1} \ \Sigma_{n-1} ^{2^nj+l}  , \\
                                    & =&  \sum \limits_{j=0}^{k-1} \Big{(}  \sum \limits_{l=0}^{2^{n-1}-1}  \Sigma_{n-1} ^{2^{n-1}(2j)+l} + \sum \limits_{l=0}^{2^{n-1}-1} \Sigma_{n-1} ^{2^{n-1}(2j+1)+l}  \Big{)} , \\
                                    & = & \sum \limits_{j=0}^{k-1} \Big{(} u_{2j}  \sum \limits_{l=0}^{2^{n-1}-1} a(n-1,l)+ u_{2j+1}  \sum \limits_{l=0}^{2^{n-1}-1} a(n-1,l) \Big{)} , \\
                                    &=&   \Big{(}   \sum \limits_{l=0}^{2^{n-1}-1} a(n-1,l) \Big{)} \cdot   \Big{(} \sum \limits_{j=0}^{k-1} u_{2j}+u_{2j+1} \Big{)} ,  \\
                                    & =& 0.
\end{array}
\]

Now, we focus on the recurrence relations verified by the coefficients $a(n,k)$. The integer $n$ is already fixed, we show this result by induction on $l$ and $k$. For $ l = $ 0, we have seen that this result was true for all integers $ k $. Suppose Equation (\ref{eq:1}) holds for all $ k $ up to a rank $ l $ and show that it is still true for all $ k $ the rank $l+1$.
\[
\begin{array}{clll}
\Sigma_{n}^{2^{n}k+l+1} &=& \Sigma_{n}^{2^{n}k+l}+\Sigma_{n-1}^{2^{n-1}(2k)+l} = a(n,l)u_{k}+a(n-1,l)u_{2k} , \\
                                              &=&  a(n,l)u_{k}+a(n-1,l)u_{k}  = \Big{(}a(n,l)+a(n-1,l)\Big{)}u_{k} .\\
\end{array}
\]
\[
\begin{array}{clll}                                             
\Sigma_{n}^{2^{n}k+2^{n-1}+l+1} &=& \Sigma_{n}^{2^{n}k+2^{n-1}+l}+\Sigma_{n-1}^{2^{n-1}(2k+1)+l} ,  \\
                                                            &=& a(n,l+2^{n-1})u_{k}+a(n-1,l)u_{2k+1} ,  \\
                                                            &= & a(n,l)u_{k}-a(n-1,l)u_{k} , \\
                                                            &=& \Big{(}a(n+2^{n-1},l)-a(n-1,l)\Big{)}u_{k}. \\
\end{array}
\]
Then, we verify the last relation of the lemma:
\[
\begin{array}{clll}
a(n+1,l+2^n) &=& a(n+1,l+2^n-1)-a(n,l-1), \\
                        &=& a(n+1,l+2^n-2)-a(n,l-2)-a(n,l-1), \\
                        &=& a(n+1,2^n )- \sum _{j=0} ^{l-1} a(n,j).
\end{array}
\]
We get $a(n+1,l+2^n)=a(n+1,2^n)-a(n+1,l) $. 
\end{proof}

\begin{lemma}
For any integer $n$, $a(n,2^{n-1}) =2^{(n-1)(n-2)/2}$.
\end{lemma}

\begin{proof}
Since $a(1,1)=1$, this result is immediate by induction from the relation:
\[
\begin{array}{clll}
a(n+1,2^n)&=&\sum \limits_{l=0}^{2^n-1}  a(n,l) = \sum  \limits _{l=0}^{2^{n-1}-1}  a(n,l) +\sum  \limits _{l=0}^{2^{n-1}-1}  a(n,l+2^{n-1} ) ,\\
                    &=& \sum  \limits _{l=0}^{2^{n-1}-1}  a(n,l) +\sum \limits  _{l=0}^{2^{n-1}-1} \Big{(} a(n,2^{n-1})-a(n,l ) \Big{)} ,\\
                    &=& \sum _{l=0}^{2^{n-1}-1} a(n,2^{n-1}) = 2^{n-1}a(n,2^{n-1}).\\
\end{array}
\]
So, $a(n+1,2^n) =  2^{n-1} \cdot 2^{(n-1)(n-2)/2} = 2^{(n+1-1)(n+1-2)/2}$.
\end{proof}

\begin{lemma} \label{le:pla}
For every integer $n$, and $l\in\{0,\dots,2^{n}-1\}$, 
\begin{equation} \label{eq:2} 0 \leq a(n,l)\leq 2^{(n-1)(n-2)/2}. \end{equation}
\end{lemma}

\begin{proof}
We will show this by induction on the integer $ n $. We initialized the recurrence.  We suppose that the result is true up to the rank $ n $ and show that it is still true to the rank $ n +1 $. 
\bigskip

Suppose then that for each integer $l\in\{0,\dots,2^{n}-1\}$, Equation (\ref{eq:2}) holds. Since for every $l\in \{0,\dots,2^{n}-1\}$, 
\[
a(n+1,l+1)=a(n,l)+a(n+1,l)\geq 0,
\] 
the sequence $\Big{(} a(n+1,l) \Big{)}_{l\in \{0,\dots,2^{n}\}}$ increases from $0$ to $2^{(n-1)(n-2)/2}$ for $l=2^n$. We can then conclude because if $l\in \{0,\dots,2^{n}-1\}$,
\[
0\leq a(n+1,l+2^n) = 2^{(n-1)(n-2)/2}-a(n+1,l)\leq 2^{(n-1)(n-2)/2}. \qedhere
\]
\end{proof}

\begin{lemma}
For every integer $n$, and $l\in\{0,\dots,2^{n-2}-1\}$, 
\[
a(n,2l+1)\geq a(n,2l)\geq 2^{n-2} a(n,l).
\]
\end{lemma}

\begin{proof}
We prove this lemma by induction on $n$. For $n = 1$, the result is immediate. We show that if the result is true up to the rank $n$, it is still true to the rank $ n+1$. We show this by induction on $ l $. From Lemma  \ref{le:lle}, this is true for $ l = 0$ and $l=1$.  We suppose that the result is true for $2l$ and $2l+1$, and we show that it is still true for $2l+2$ and $2l+3$.
\bigskip

\noindent
\underline{If $l\in\{0,\dots,2^{n-2}-1,\}$}, then $a(n,2l)\leq a(n,2l+1)$ and
\[
\begin{array}{clll}
a(n+1,2(l+1)+1)&\geq &a(n+1,2(l+1)) \\
                            &\geq& a(n+1,2l)+a(n,2l)+a(n,2l+1) \\
                            &\geq& 2^{n-1}a(n,l)+a(n,2l)+a(n,2l+1) \\
                            &\geq& 2^{n-1}a(n,l)+2a(n,2l) \\
                            &\geq& 2^{n-1}a(n,l)+2 2^{n-2}a(n-1,l) \\
                            &\geq&  2^{n-1}\Big{(} a(n,l)+a(n-1,l) \Big{)}  \\
                            &\geq &2^{n-1} a(n,l+1) .\\
\end{array}
\]
\underline{If $l\in\{2^{n-2},\dots,2^{n-1}-1,\}$}, then $a(n,2l)\geq a(n,2l+1)$ and
\[
\begin{array}{clll}
a(n+1,2(l+1)+1) &\geq& a(n+1,2(l+1)) \\
                            &\geq& a(n+1,2l)+a(n,2l)+a(n,2l+1) \\
                            &\geq& 2^{n-1}a(n,l)+a(n,2l)+a(n,2l+1) \\
                            &\geq& 2^{n-1}a(n,l)+2a(n,2l+1)\\
                            &\geq& 2^{n-1}a(n,l)+2 2^{n-2}a(n-1,l) \\
                            &\geq&  2^{n-1}\Big{(} a(n,l)+a(n-1,l) \Big{)}  \\
                            &\geq &2^{n-1} a(n,l+1) .\\
\end{array}
\]
\end{proof}

\section{Proof of Theorem \ref{th}}

Let us start by proving the following lemma.

\begin{lemma}
\label{le:3}
Let $n$ be an integer greater than or equal to $1$.
\begin{enumerate}
\item For each integer $m$, $f_n(2m+1)=u_m$ and $f_n(2m)=0$.
\item For each real $x$, $f_n(x)\in[-1,1]$.
\item For each integer $m$ and for each  $x\in[0,2]$,  
\begin{equation} \label{eq:3} f_n(x+2m)=-f_n(x)u_m. \end{equation}
\item For each integer $m$, if $u_m=-1$, $f_n$ increases on $[m,m+1]$, and if $u_m=1$, $f_n$ decreases on $[m,m+1]$. In particular, $f_n$ and $u_m$ have the same sign on $[2m,2m+2]$.
\item For each couple of reals $(x,y)\in[0,2]^2$: $|f_n(x)-f_n(y)| \leq | x-y|$.
\item For each real $x\in[0,1]$, the sequence $ ( f_n(x))_n$ decreases.
\item For each real $x\in[2m,2m+1]$, $ ( f_n(x))_n$ is decreasing if $u_m=-1$, and increasing otherwise. And for each real $x\in[2m+1,2m+2]$, $ ( f_n(x))_n$ is increasing if $u_m=-1$, and decreasing otherwise
\end{enumerate}
\end{lemma}

\begin{proof}[Proof of Point 1 of Lemma \ref{le:3}]
We fix an integer $n\geq1$, and an integer $m$. By the definition of functions $f_n$, $f_n(2m) = \Sigma_n^{2^{n}m}  2^{-(n-1)(n-2)/2}=0$ and
\[
f_n(2m) = \Sigma_n^{2^{n}m+2^{n-1}}  2^{-(n-1)(n-2)/2}=2^{(n-1)(n-2)/2}u_m2^{-(n-1)(n-2)/2}=u_m. \qedhere
\]
\end{proof}

\begin{proof}[Proof of Point 2 of Lemma \ref{le:3}]
From Lemma \ref{le:pla}, for each positive real $x$:
\[
\begin{array}{clll}
| f_n(x) | & \leq& \sup\left \{ |  x^{2^nk +l}_n | ;k\in \mathbb N\mbox{ and }l\in \{0,\dots,2^n\}  \right\} \\
               &\leq &\sup\left \{ |   \Sigma ^{2^nk +l}_n2^{-(n-1)(n-2)/2} | ;k\in \mathbb N\mbox{ and }l\in \{0,\dots,2^n\}  \right\} \\
               & \leq& \sup\left \{ a(n,l) 2^{-(n-1)(n-2)/2} ; l\in \{0,\dots,2^n\}  \right\}. \\
\end{array} 
\]
From Lemma \ref{le:pla}, $|a(n,l) 2^{-(n-1)(n-2)/2} |\leq 1$ and $|f_n(x)|\leq1$.
\end{proof}

\begin{proof}[Proof of Point 3 of Lemma \ref{le:3}]
We fix a real $x=\frac{k}{2^{n-1}}+\delta_x\in[0,2)$, and an integer $m$. 
\[
\renewcommand{\arraystretch}{1.4}
\begin{array}{clll}
f_n(x+2m)&=&f_n \left( \frac{k+m2^{n}}{2^{n-1}}+\delta_x \right) = x^{k+m2^n}_n+\delta_x(x^{k+m2^n+1}_n-x^{k+m2^n}_n)\\
&=&-u_m\Big{(}x^{k+m2^n}_n+\delta_x(x^{k+m2^n+1}_n-x^{k+m2^n}_n)\Big{)}=-u_m f_n(x).
\end{array}
\]

We treat now the case where $x = 2m$. From Point 1, $ f_n (2m) = f_n(0)=0 $ and this point is demonstrated.
\end{proof}

\begin{proof}[Proof of Point 4 of Lemma \ref{le:3}]
We verify this result by induction on $ n $. For $n= 1$, the result is true. Now, we verify that if  it is true up to the rank $n-1$, it will be still true to the rank $n$. For any integer $k$,
\[
\renewcommand{\arraystretch}{1.4} 
\begin{array}{clll}
 f_n\left( \frac{k}{2^{n-1}} \right) - f_n\left( \frac{k+1}{2^{n-1}}\right )  &=& x^{k}_n-x^{k+1}_n  = ( \Sigma^{k}_n-\Sigma^{k+1}_n ) 2^{-(n-1)(n-2)/2} ,\\
                                                                                                                      &=& \Sigma ^{k}_{n-1} 2^{-(n-1)(n-2)/2} =  \Sigma ^{k}_{n-1} 2^{-(n-1-1)(n-1-2)/2}2^{-n+1}, \\
                                                                                                                      & = & x^{k}_{n-1} 2^{-n+1} = f_{n-1}(\frac{k}{2^{n-2}}) 2^{-n+1}.
\end{array}
\]

So, if $k=m2^{n-1}+l$, with $0\leq l\leq 2^{n-1}-1$, then $\frac{k}{2^{n-2}} = \frac{l}{2^{n-2}}+2m$, and $f_{n-1}(\frac{k}{2^{n-2}})$ and $u_m$ have the same sign. Then,  $f_n$ is decreasing on $[m,m+1]$ if $u_m=-1$, and increasing otherwise.

And if $k=(m+1)2^{n-1}+l$, with $0\leq l\leq 2^{n-1}-1$, then $\frac{k}{2^{n-2}} = \frac{l}{2^{n-2}}+2(m+1)$, and $f_{n-1}(\frac{k}{2^{n-2}})$ have the same sign as  $u_{m+1}$, and so  $f_n$ is decreasing on $ [m, m +1]$ if $u_{m+1}=-1$, and increasing otherwise.

We suppose now that $u_m=u_{2m}=-1$. The function $f_n$ decreases from $0$ to $-1$ on $[2m,2m+1]$, and increases from $-1$ to $0$ on $[2m+1,2m+2]$. So, the function is negative on $[2m,2m+2]$.  We can then use the same argument if $u_m = 1$ to complete the proof of this point.
\end{proof}

\begin{proof}[Proof of Point 5 of Lemma \ref{le:3}]
We fix two reals $x$ and $y$, such that  $x\leq y$, verifying:
\[
x=\frac{k}{2^{n-1}}+\delta_x \mbox{ and } y=\frac{l}{2^{n-1}}+\delta_y,
\]
where $\delta_x$ and $\delta_y$ are less than $1/2^{n-1}$.
\[
\renewcommand{\arraystretch}{1.4}
\begin{array}{l}
f_n(x)-f_n(y) =  x^k_n +2^{n-1}\delta_x (x^{k+1}_n-x^k_n) -  x^l_n -2^{n-1}\delta_y (x^{l+1}_n-x^l_n),\\
\begin{array}{llll}
& = & x^k_n-x^{k+1}_n+\dots+x^{l-1}_n+x^{l}_n +2^{n-1}\delta_x (x^{k+1}_n-x^k_n) -  2^{n-1}\delta_y (x^{l+1}_n-x^l_n) ,\\
&=& \left( f_{n-1} \left( \frac{k}{2^{n-2}} \right) + \dots + f_{n-1}\left( \frac{l-1}{2^{n-1}} \right)  \right) 2^{1-n}-\delta_x  f_{n-1} \left( \frac{k}{2^{n-2}} \right) +\delta_y f_{n-1} \left( \frac{l}{2^{n-2}} \right) .
\end{array}
\end{array}
\]

Since $f_{n-1}$ is negative on $[0,1]$:
\[
| f_n(x)-f_n(y) | \leq \frac{l-k}{2^{n-1}}+ \delta_y-\delta_x = | x-y |. \qedhere
\]
\end{proof}

\begin{proof}[Proof of Point 6 of Lemma \ref{le:3}]
We show that for each integer $N$ and each integer $l\in\{0,\dots,2^{n-1}-1\}$, the sequence $\left( f_{n+N+1} \left( \frac{l}{2^{N-1}} \right)\right)_n$ is decreasing:
\[
\renewcommand{\arraystretch}{1.4}
\begin{array}{clll}
f_{n+N+1} (l/2^{N-1} )&=& f_{n+N+1} \left( \frac{l2^n}{2^{N+n-1}} \right) = f_{n+N+1} \left( \frac{l2^n-1}{2^{N+n-1}} \right) = x^{l2^{n}}_{N+n+1}, \\
                                       &=& \Sigma ^{l2^n} _{N+n+1} 2^{(N+n)(N+n-2)} =  a(N+n+1,l2^n)2^{(N+n)(N+n-2)} u_0, \\
                                       &=& - a(N+n+1,l2^n)2^{(N+n)(N+n-2)}, \\
                                       & \leq& -2^{N+n-1} a(N+n,l2^{n-1}) 2^{(N+n)(N+n-1)/2}.
\end{array}
\]
This result is then proved because 
\[
\begin{array}{c}
-2^{N+n-1} a(N+n,l2^{n-1}) 2^{(N+n)(N+n-1)/2} =  -a(N+n,l2^{n-1}) 2^{(N+n-1)(N+n-2)/2}, \\
\begin{array}{lcl} & =&  a(N+n,l2^{n-1}) 2^{(N+n-1)(N+n-2)/2} u_0, \\ & =& f_{n+N}(l/2^{N-1}).  \end{array}
\end{array}
\]
Then, $f_{N+n+1} (l/2^{N-1} )\leq  f_{N+n}(l/2^{N-1})$.
\end{proof}

\begin{proof}[Proof of Point 7 of Lemma \ref{le:3}]
Let $x\in [0,1]$ and $m$ be an integer. We deduce the proof of this point from the following remarks:
\[
f_n(x+1) =-1-f_n(x)  \mbox{ and } f_n(x+2m)=-u_m \cdot f_n(x). \qedhere
\]
\end{proof}

\begin{proof}[Proof of Theorem \ref{th}]
For each real number $x$, the sequence $(f_n(x))_n$  is monotone and bounded, so it converges. Let $f_\infty(x)$ denote the limit. It is  clear that the function $ f_ \infty $ is $1$-Lipschitz. The third point is proved from Equation (\ref{eq:3}). With the previous lemma, we can deduce that the range of the function $ f_\infty $ is included in $[-1,1] $ and that for each positive integer $m$: 
\[
 f_\infty (2m) = 0 \mbox{ and } f_\infty (2m +1) =  u_m.
\] 
 
We need to verify that it is a solution of  Equation (\ref{eq:0}). We fix a positive real $X\in[2m,2m+2[$ and an integer $n$  such that $[X,X+1/2^{n-1}]\subset]2m,2m+2[$. We fix $l$, the integer such that $0 \leq \delta_X = X-l/2^{n-1}\leq 1/2^{n-1}$.
\bigskip

Then for any integer $ m $ sufficiently large:
\[
\begin{array}{cllll}
f_{n+m+1}(X)=  &=& f_{n+m+1} \left( \frac{l2^{m}}{2^{n+m-1}} \right)+ f_{n+m+1}(X) -  f_{n+m+1} \left( \frac{l}{2^{n-1}} \right) , \\
                            &=& 2^{1-n-m} \sum \limits_{j=0} ^{l2^{m}-1} f_{n+m}  \left( \frac{j}{2^{n+m}} \right)  + f_{n+m+1}(X) -  f_{n+m+1} \left( \frac{l}{2^{n-1}} \right).\\
\end{array}
\]
 
Since for each real $x$, the sequence $(f_n (x)) _n$ is monotone, we let $m$ tend to infinity to find:
\[
f_\infty (X) =  \int _{0} ^{2\frac{l}{2^{n-1}} } f_\infty(x)dx + f_\infty(X) - f_\infty \left( \frac{l}{2^{n-1}} \right).
\]

We deduce therefore that 
\[
\left | f_\infty (X) -  \int _{0} ^{2X} f_\infty(x)dx \right| \leq 2\left |   f_\infty(X) - f_\infty \left( \frac{l}{2^{n-1}} \right) \right | 
\leq \frac{1}{2^{n-2}}.
\]

Then, when $n$ goes to infinity, $f_\infty (X) =  \int _{0} ^{2X} f_\infty(x)dx$.
\end{proof}

\end{document}